\documentclass[11pt, reqno]{amsart}
\usepackage{amsmath, amsthm, amscd, amsfonts, amssymb, graphicx, color, mathrsfs}
\usepackage[bookmarksnumbered, colorlinks, plainpages=false]{hyperref}
\usepackage[all]{xy}
\usepackage[utf8]{inputenc}
\usepackage{slashed}
\usepackage{cleveref}
\usepackage{amsmath}
\usepackage{mathtools}
\usepackage{enumerate}
\usepackage{tgadventor}
\usepackage[
backend=biber,
style=alphabetic,
sorting=ynt
]{biblatex}
\newcommand{\beq}[1]{\begin{equation}\label{eq:#1}}
	\newcommand{\eeq}{\end{equation}}

\newtheorem{theorem}{Theorem}[section]

\newtheorem{notation}[theorem]{Notation}

\newtheorem{proposition}[theorem]{Proposition}

\theoremstyle{definition}
\newtheorem{definition}[theorem]{Definition}

\theoremstyle{remark}
\newtheorem{remark}[theorem]{Remark}
\numberwithin{equation}{section}

\def\G{{\mathbb{G}}}
\def\R{{\mathcal{R}}}
\def\L{{^{2}_{L^2(\G)}}}

\newcommand{\vertiii}[1]{{\left\vert\kern-0.25ex\left\vert\kern-0.25ex\left\vert #1 
    \right\vert\kern-0.25ex\right\vert\kern-0.25ex\right\vert}}
\allowdisplaybreaks

\begin{document}

	\setcounter{page}{1}
	
\title[Hypoelliptic Klein-Gordon equation with singular mass]{Fractional Klein-Gordon equation with singular mass. II: Hypoelliptic case}
	
\author[M. Chatzakou]{Marianna Chatzakou}
\address{
	Marianna Chatzakou:
	\endgraf
    Department of Mathematics: Analysis, Logic and Discrete Mathematics
    \endgraf
    Ghent University, Belgium
  	\endgraf
	{\it E-mail address} {\rm marianna.chatzakou@ugent.be}
		}

\author[M. Ruzhansky]{Michael Ruzhansky}
\address{
  Michael Ruzhansky:
  \endgraf
  Department of Mathematics: Analysis, Logic and Discrete Mathematics
  \endgraf
  Ghent University, Belgium
  \endgraf
 and
  \endgraf
  School of Mathematical Sciences
  \endgraf
  Queen Mary University of London
  \endgraf
  United Kingdom
  \endgraf
  {\it E-mail address} {\rm michael.ruzhansky@ugent.be}
  }

\author[N. Tokmagambetov]{Niyaz Tokmagambetov}
\address{
  Niyaz Tokmagambetov:
  \endgraf
  Department of Mathematics: Analysis, Logic and Discrete Mathematics
  \endgraf
  Ghent University, Belgium
  \endgraf
  and
    \endgraf   
    Al--Farabi Kazakh National University
    \endgraf
    Almaty, Kazakhstan
    \endgraf
  {\it E-mail address} {\rm niyaz.tokmagambetov@ugent.be, tokmagambetov@math.kz}
  }

\thanks{The authors are supported by the FWO Odysseus 1 grant G.0H94.18N: Analysis and Partial Differential Equations and by the Methusalem programme of the Ghent University Special Research Fund (BOF) (Grant number 01M01021). Michael Ruzhansky is also supported by EPSRC grant EP/R003025/2. }

\dedicatory{Dedicated to Vladimir Rabinovich on the occasion of his 80th birthday}

\begin{abstract} 
In this paper we consider a fractional wave equation for hypoelliptic operators with a singular mass term depending on the spacial variable and prove that it has a very weak solution. Such analysis can be conveniently realised in the setting of graded Lie groups. The uniqueness of the very weak solution, and the consistency with the classical solution are also proved, under suitable considerations.
This extends and improves the results obtained in the first part \cite{ARST21a} which was devoted to the classical Euclidean Klein-Gordon  equation.
\end{abstract}

\maketitle

\tableofcontents

\section{Introduction}
The aim of this paper is to contribute to the study of the Klein-Gordon equation for positive (left) Rockland operator $\mathcal{R}$  (left-invariant hypoelliptic partial differential operator which is homogeneous of positive degree $\nu$) on a general graded Lie group $\mathbb{G}$, with a possibly singular mass term depending on the spacial variable; that is for $T>0$, and for $s>0$ we consider the Cauchy problem
\begin{equation}
\label{kg.eq}
\begin{cases}
u_{tt}(t,x) +\mathcal{R}^{s}u(t,x)+m(x)u(t,x)=0\,,\quad (t,x)\in [0,T]\times \mathbb{G}\,,\\
u(0,x)=u_0(x)\,,u_{t}(0,x)=u_1(x),\; x \in \G\,,	
\end{cases}       
\end{equation}
where $m$ is a non-negative and possibly singular function/distribution. 

The setting of Rockland operators on graded Lie groups allows one to consider both elliptic and subelliptic settings in \eqref{kg.eq}. The simplest example is that of the standard Klein-Gordon equation, when we take $\G=\mathbb R^d$ to be the Euclidean space, and $\mathcal{R}=-\Delta$ to be the Laplacian on $\mathbb R^d$. However, already on $\mathbb R^d$, the setting of \eqref{kg.eq} allows one to consider more general evolutions, for example, taking
$$
\mathcal{R}= (-1)^m\sum_{j=1}^d\frac{\partial^{2m}}{\partial x_j^{2m}},
$$ 
for any integer $m$. Such operators are also Rockland operators on $\mathbb R^d$, as we explain in the next section. However, the general setting of \eqref{kg.eq} allows one to also consider hypoelliptic operators. The simplest example would be $\G$ being the Heisenberg group, and $\mathcal{R}$ the positive sub-Laplacian on it. More generally, if $\G$ is any stratified group (or a homogeneous Carnot group), and $X_1,\ldots,X_N$ are the generators of its Lie algebra (satisfying the H\"ormander condition), we can consider 
$$
\mathcal{R}= (-1)^m\sum_{j=1}^N X_j^{2m},
$$ 
for any integer $m$, where we understand $X_j$ also as the derivative with respect to the vector field $X_j$.

The main feature of \eqref{kg.eq} is that we will not assume any regularity on the mass coefficient $m$. Especially, we are interested in irregular $m$, for example being $\delta$-distribution, or even $\delta^2$, if understood appropriately in the sense of multiplication of distributions. We note that in this situation the usual notion of weak solutions is not applicable to \eqref{kg.eq} in view of the Schwartz impossibility result 
\cite{Sch54} on the multiplication of distributions.

Thus, in this paper we work with the concept of very weak solutions. More specifically, we will show its applicability to the Cauchy problem \eqref{kg.eq} for the Klein-Gordon equation for the Rockland operator $\mathcal{R}$ on the graded Lie group $\mathbb{G}$ with a singular mass depending on the spacial variable. This concept was introduced in \cite{GR15} to deal with the Schwartz impossibility result about multiplication of distributions \cite{Sch54}, in the context of wave type equations with singular coefficients. Later, this analysis was applied to other hyperbolic type equations with singular coefficients \cite{RT17a}, \cite{RT17b}, \cite{ART19}, and \cite{MRT19}. 
The wave type equations with time-dependent coefficients on graded Lie groups were analysed in \cite{RT21} for H\"older coefficients, and in \cite{RY20} for distributional time-dependent coefficients, using the notion of very weak solutions. 
All these works deal with the time-dependent equations and in the recent papers \cite{Gar20}, \cite{ARST21a}, \cite{ARST21b}, and \cite{ARST21c}, the authors start to develop the notion of very weak solutions for equations with (irregular) space-depending coefficients.

The present paper is the extension and improvement of the results obtained in the first part \cite{ARST21a} which was devoted to the classical Klein-Gordon equation. In fact, the setting of \cite{ARST21a} was the equation \eqref{kg.eq} for $\G=\mathbb R^d$ and $\mathcal{R}=-\Delta$ being the positive Laplacian on the Euclidean space. Consequently, the results here contain the results of \cite{ARST21a} as a special case, and we also use this chance to slightly correct the consistency statement given in that paper, see Remark \ref{finrem}, as well as a clarifying Remark \ref{rem1p}.

\section{Preliminaries}
Let us briefly recall some basic concepts, terminology and notation on graded Lie groups that will be useful for the ideas we develop throughout this paper. For a more detailed exposition we refer to Folland and Stein [Chapter 1 in \cite{FS82}], or, to the more recent open access book, by Fischer and the second author [Chapter 3 in \cite{FR16}]. 
\par Let $\G$ be a nilpotent Lie group, and let $\mathfrak{g}$ be its Lie algebra. Its lower series is the descending sequence $\{\mathfrak{g}_i\}$ of ideals defined inductively by $\mathfrak{g}_1=\mathfrak{g}$, $\mathfrak{g}_i=[\mathfrak{g},\mathfrak{g}_{i-1}]$, for $i>1$. If $\mathfrak{g}$ admits a gradation of vector spaces as $\mathfrak{g}= \bigoplus_{i=1}^{\infty}  \mathfrak{g}_{i}$, where all, but finitely many $\mathfrak{g}_i$'s are equal to $\{0\}$, and is such that $[\mathfrak{g}_{i},\mathfrak{g}_{j}]\subset \mathfrak{g}_{i+j}$, for all $i,j$, then $\G$ is a \textit{graded Lie group}. Graded Lie groups are naturally \textit{homogeneous Lie groups}; that is $\mathfrak{g}$ is equipped with a one-parameter family $\{D_r\}_{r>0}$ of automorphisms of $\mathfrak{g}$ of the form $D_r=\exp(A\,\text{log} r)$, with $A$ being a diagonalisable linear operator on $\mathfrak{g}$ with positive eigenvalues. Such automorphisms shall be called \textit{dilations}. 
\par We have the following nested subclasses of Lie groups:
\[
\text{nilpotent}\supset \text{homogeneous} \supset \text{graded} \supset \text{stratified} \supset \{\text{Heisenberg, Engel, Cartan}\}\,.
\]
The cases of the Heisenberg, Engel and Cartan groups, are examples of graded Lie groups whose associated representation theory is well-understood in the sense that there exists a complete and explicit classification of the unitary, irreducible representations on them; see e.g. \cite{Tay84}, \cite{Dix57}, as well as the analysis in \cite{Cha20}, \cite{Cha21}. For graded Lie algebras $\mathfrak{g}$ 
of dimension $n$, 
the \textit{canonical family of dilations}, is the one dictated by the gradation of $\mathfrak{g}$, and is given by 
\begin{equation}\label{dw}
 X_{i}^{(j)}\circ D_r=r^{v_i}D_r \circ X_{i}^{(j)},
\end{equation}
where $X_{i}^{(j)} \in \mathfrak{g}_j$, $i=1,\ldots,n,$
and $v_i$'s are the same for all vectors $X_{i}^{(j)} \in \mathfrak{g}_j$.
These $v_i$'s are called the \textit{dilations' weights}.
\par In the case of graded Lie groups, or more generally in the case of nilpotent Lie groups, the exponential map (on the group) is a diffeomorphism from $\mathfrak{g}$ onto $\G$, under the group law that has been defined accordingly to the structure of $\mathfrak{g}$ due to the Baker–Campbell–Hausdorff formula; see, e.g. \cite{CG90}. More generally, this identification allows for the transmission of ideas from the infinitesimal level of the Lie algebra $\mathfrak{g}$ to the level of the group $\G$. Additionally, when $\mathfrak{g}$ is homogeneous, then, the dilations can be transported to the group side, while the Lebesgue measure $dx$ on $\mathfrak{g}$ is the bi-invariant \textit{Haar measure} on $\G$, and the number $Q$ that satisfies $d(D_r(x))=r^Q\,dx$, that is the sum of the eigenvalues of the matrix $A$, shall be called the \textit{homogeneous dimension} of $\G$. 
\par On the other hand, any element $\pi \in \widehat{\G}$ of the unitary dual of $\G$, with $\pi$ acting on some separable Hilbert space $\mathcal{H}_\pi$, gives rise to the representation $d\pi$ on the space of smooth vectors $\mathcal{H}_{\pi}^{\infty}$ on the infinitesimal level; that is we can define
\[
d\pi(X)v:=\lim_{t \rightarrow 0} \frac{1}{t}(\pi(exp(tX))v-v)\,,\quad X \in \mathfrak{g}\,,v \in \mathcal{H}_{\pi}^{\infty}\,.
\]
The above definition, due to the Poincar\'{e}-Birkhoff-Witt Theorem (see, e.g. \cite{Bir37}, see also a discussion in \cite{FR16}), that identifies that space of left-invariant operators in $\mathfrak{g}$ with the universal enveloping Lie algebra  $\mathfrak{U}(\mathfrak{g})$, can be extended to any $T \in \mathfrak{U}(\mathfrak{g})$, i.e., we can write $d\pi(T)$; or, with an abuse of notation, $\pi(T)$.
\par A remarkable class among left-invariant operators, that generalises the notion of the sub-Laplacian on the bigger class of graded groups, is that of \textit{Rockland} operators, which are usually denoted by $\R$. The latter is a class of operators that are hypoelliptic on $\G$ \cite{HN79}, and homogeneous of a certain positive degree. 
So, {\em by Rockland operators we understand the homogeneous left-invariant hypoelliptic differential operators on $\G$.}
For additional characterisations of the Rockland operators, we refer to \cite{Roc78}, \cite{Bea77}, \cite{TR97}, as well as to a presentation in \cite{FR16}.
\par We recall that $\R$ and $\pi(\R)$, are densely defined on their domains $\mathcal{D}(\G) \subset L^2(\G)$, and $\mathcal{H}_{\pi}^{\infty}\subset \mathcal{H}_\pi$, respectively (cf. [Proposition 4.1.15 in \cite{FR16}]. The latter implies that the positivity of $\R$, as required for our purposes, amounts to the condition
\[
(\R f,f)_{L^2(\G)}\geq 0\,,\quad f \in \mathcal{D}(\G)\,.
\]
 We remark that, for a positive Rockland operator $\R$, the spectrum of the operator $\pi(\R)$, with $\pi \in \widehat{\G}\setminus \{1\}$, is discrete \cite{HJL85}, which allows us to choose an orthonormal basis for $\mathcal{H}_\pi$ such that the self-adjoint operator $\pi(\R)$ can be identified with the infinite dimensional matrix with diagonal elements $\pi_{k,k}^{2}\equiv \pi_{k}^{2}$, with $\pi_k \in \mathbb{R}_{+}$. 
\par Let us now recall that the \textit{group Fourier transform} of a function $f \in L^1(\G)$ at $\pi \in \widehat{\G}$ is the bounded operator $\widehat{f}(\pi)$ (often denoted by $\pi(f)$) on $\mathcal{H}_\pi$ given by 
\[
(\widehat{f}(\pi)v_1,v_2)_{\mathcal{H}_\pi}:=\int_{\G} f(x)(\pi^{*}(x)v_1,v_2)_{\mathcal{H}_\pi}\,dx\,,\quad v_1,v_2 \in \mathcal{H}_\pi\,.
\]
\par If $f \in L^2(\G) \cap L^1(\G)$, then $\widehat{f}(\pi)$ is a Hilbert-Schmidt operator, and we have the following isometry, known as the \textit{Plancherel formula}
	\begin{equation}\label{planc.id}
		\int_{\G} |f(x)|^2\,dx=\int_{\widehat{{\G}}}\|\pi(f)\|^{2}_{\textnormal{HS}}\,d\mu(\pi)\,,
		\end{equation}
		where $\mu$ stands for the \textit{Plancherel measure} on $\G$. For a detailed exposition of the Plancherel Theorem and the relevant theory, we refer to \cite{Dix77, Kir07, CG90}, or to [Section 1.8, Appendix B.2 in \cite{FR16}].
\par Finally, since the action of a Rockland operator $\R$ is involved in our analysis, let us make a brief overview of some related properties.
\begin{definition}[Homogeneous Sobolev spaces]
For $s>0$, $p>1$, and $\R$ a positive homogeneous Rockland operator of degree $\nu$, we define the $\R$-Sobolev spaces as the space of tempered distributions $\mathcal{S}^{'}(\G)$ obtained by the completion of $\mathcal{S}(\G)\cap \textnormal{Dom}(\R^{\frac{s}{\nu}})$ for the norm
\[
\|f\|_{\dot{L}^{p}_{s}(\G)}:=\|\R^{\frac{s}{\nu}}_{p}f\|_{L^p(\G)}\,,\quad f \in \mathcal{S}(\G)\cap \textnormal{Dom}(\R^{\frac{s}{\nu}}_{p})\,,
\]
where $\R_p$ is the maximal restriction of $\R$ to $L^{p}(\G)$.\footnote{When $p=2$, we will write $\R_2=\R$ for the self-adjoint extension of $\R$ on $L^2(\G)$.}
\end{definition}
Let us mention that, the above $\R$-Sobolev spaces do not depend on the specific choice of $\R$, in the sense that, different choices of the latter produce equivalent norms, see [Proposition 4.4.20 in \cite{FR16}]. 
\par In the scale of these Sobolev spaces, we recall the next proposition as in [Proposition 4.4.13 in \cite{FR16}].
\begin{proposition}[Sobolev embeddings]
For $1<\tilde{q}_0<q_0<\infty$ and for $a,b \in \mathbb{R}$ such that 
\[
b-a=Q \left( \frac{1}{\tilde{q}_0}-\frac{1}{q_0}\right)\,,
\]
we have the continuous inclusions 
\[
\dot{L}^{\tilde{q}_0}_{b}(\G) \subset \dot{L}^{q_0}_{a}(\G)\,,
\]
that is, for every 
$f \in \dot{L}^{\tilde{q}_0}_{b}(\G)$, we have $f \in \dot{L}^{q_0}_{a}(\G)$, and there exists some positive constant $C=C(\tilde{q}_0,q_0,a,b)$ (independent of $f$) such that 
\begin{equation}
\label{inclusions}
\|f\|_{\dot{L}^{q_0}_{a}(\G)}\leq C \|f\|_{\dot{L}^{\tilde{q}_0}_{b}(\G)}\,.
\end{equation}
\end{proposition}
\par In the sequel we will make use of the following notation:
\begin{notation}
\begin{itemize}
    \item When we write $a \lesssim b$, we will mean that there exists some constant $c>0$ (independent of any involved parameter) such that $a \leq c b$;
    \item if $\alpha=(\alpha_1,\cdots,\alpha_n) \in \mathbb{N}^n$ is some multi-index, then we denote by 
    \[
    [\alpha]=\sum_{i=1}^{n} v_i\alpha_i\,,
    \]
    its \textit{homogeneous length}, where the $v_i$'s stand for the dilations' weights as in \eqref{dw}, and by 
    \[
    |\alpha|=\sum_{i=1}^{n}\alpha_i\,,
    \]
    the length of it;
    \item for suitable $f \in \mathcal{S}^{'}(\G)$  we have introduced the following norm
    \[
    \|f\|_{H^{s}(\G)}:=\|f\|_{\dot{L}^{2}_{s}(\G)}+\|f\|_{L^2(\G)}\,;
    \]
    \item when regulisations of functions/distributions on $\G$ are considered, they must be regarded as arising  via  convolution  with  Friedrichs-mollifiers;  that is, $\psi$ is a Friedrichs-mollifier, if it is a compactly supported smooth function with $\int_{\G}\psi\,dx=1$. Then the regularising net is defined as
\begin{equation}\label{mol}
	\psi_{\epsilon}(x)=\epsilon^{-Q}\psi(D_{\epsilon^{-1}}(x))\,,\quad \epsilon \in (0,1]\,,
\end{equation}
where $Q$ is the homogeneous dimension of $\G$.
\end{itemize}
\end{notation}

\section{Estimates for the classical solution}

Here and thereafter, we consider a fixed power $s>0$ of a fixed, positive (in the operator sense) Rockland operator $\R$ that is assumed to be of homogeneous degree $\nu$. Moreover, the coefficient $m$ in \eqref{kg.eq} will be regarded to be non-negative on $\G$.
\par The next two propositions prove the existence and uniqueness of the classical solution to the Cauchy problem \eqref{kg.eq}, in the cases where the coefficient $m$ is such that $m \in L^{\infty}(\G)$ or $m \in L^{\frac{2Q}{\nu s}}(\G)$, where, in the second case, we must additionally require $Q> \nu s$.
\begin{proposition}\label{prop1.clas.kg}
Let $m \in L^{\infty}(\G)$, $m \geq 0$, and suppose that $(u_0,u_1) \in H^{\frac{s \nu}{2}}(\G)\times L^2(\G)$. Then, there exists a unique solution $u \in C([0,T];H^{\frac{s \nu}{2}}(\G))\cap C^{1}([0,T];L^{2}(\G))$ to the Cauchy problem \eqref{kg.eq}, that satisfies the estimate 
	\begin{equation}\label{prop1.claim}
		\|u(t,\cdot)\|_{H^{\frac{s \nu}{2}}(\G)}+\|\partial_{t}u(t,\cdot)\|_{L^2(\G)} \lesssim (1+\|m\|_{L^{\infty}(\G)})\cdot \{\|u_1\|_{L^2(\G)}+\|u_0\|_{H^{\frac{s \nu}{2}}(\G)}\}\,,
		\end{equation}
		uniformly in $t \in [0,T]$.
\end{proposition}
\begin{proof}
Multiplying the equation \eqref{kg.eq} by $u_t$ and integrating over $\mathbb{G}$, we get
	\begin{equation}\label{Re0,kg}
		\Re(\langle u_{tt}(t,\cdot),u_t(t,\cdot)\rangle_{L^2(\mathbb{G})}+\langle \mathcal{R}^{s}u(t,\cdot),u_t(t,\cdot)\rangle_{L^2(\G)}+\langle m(\cdot)u(t,\cdot),u_t(t,\cdot) \rangle_{L^2(\G)})=0\,,
	\end{equation}
for all $t \in [0,T]$. It is easy to check that 
\[
\Re(\langle u_{tt}(t,\cdot),u_t(t,\cdot)\rangle_{L^2(\mathbb{G})})=\frac{1}{2}\partial_{t} \langle u_t(t,\cdot),u_t(t,\cdot)\rangle_{L^2(\G)}\,,
\]
\[
\Re(\langle \mathcal{R}^{s}u(t,\cdot),u_t(t,\cdot)\rangle_{L^2(\G)})=\frac{1}{2} \partial_{t}\langle \R^{\frac{s}{2}}u(t,\cdot), \R^{\frac{s}{2}}u(t,\cdot) \rangle_{L^2(\G)}\,,
\]
and
\[
\Re(\langle m(\cdot)u(t,\cdot),u_t(t,\cdot) \rangle_{L^2(\G)})=\frac{1}{2} \partial_{t} \langle \sqrt{m}(\cdot)u(t,\cdot),\sqrt{m}(\cdot),u(t,\cdot) \rangle_{L^2(\G)}\,.
\]
Denoting by $$E(t):=\|u_t(t,\cdot)\|\L+\|\R^{\frac{s}{2}}u(t,\cdot)\|\L+\|\sqrt{m}(\cdot)u(t,\cdot)\|\L,$$ the energy functional estimate of the system \eqref{kg.eq}, the equation \eqref{Re0,kg} implies that $\partial_{t}E(t)=0$, and consequently also that $E(t)=E(0)$, for all $t \in [0,T]$. By taking into consideration the estimate 
\begin{equation}\label{plog.mu0}
\|\sqrt{m}(\cdot)u_0\|\L \leq \|m\|_{L^\infty(\G)}\|u_0\|\L\,,
\end{equation}
by the above, it follows that each positive term that $E(t)$ consists of, is bounded itself. That is, we have that
\begin{eqnarray}
	\label{prop.1.est.m}
	\|\sqrt{m}(\cdot)u(t,\cdot)\|\L \lesssim \|u_1\|\L+\|\R^{\frac{s}{2}}u_0\|\L+\|m\|_{L^\infty(\G)}\|u_0\|\L,
\end{eqnarray}
while also that
\begin{eqnarray}\label{forRs2}
	\|u_t(t,\cdot)\|\L\,,\|\R^{\frac{s}{2}}u(t,\cdot)\|\L & \lesssim & \|u_1\|\L+\|u_0\|^{2}_{H^{\frac{s \nu}{2}}(\G)}+\|m\|_{L^\infty(\G)}\|u_0\|^{2}_{L^2(\G)}\nonumber\\
	& \lesssim & (1+\|m\|_{L^\infty(\G)})\{\|u_1\|\L+\|u_0\|^{2}_{H^{\frac{s \nu}{2}}(\G)}\},
\end{eqnarray}
uniformly in $t \in [0,T]$, where we use 
\[
\|\R^{\frac{s}{2}}u_0\|_{L^2(\G)}\,,\|u_0\|_{L^2(\G)}\leq \|u_0\|_{H^{\frac{s \nu}{2}}(\G)}\,.
\]
Observe that, to prove \eqref{prop1.claim}, it remains to show the desired estimate for the norm $\|u(t,\cdot)\|_{L^2(\G)}$. To this end, we first apply the group Fourier transform to \eqref{kg.eq} with respect to $x \in \G$ and for all $\pi \in \widehat{\G}$, and we get 
\begin{equation}
	\label{prop1.ft}
	\begin{cases}
	\widehat{u}_{tt}(t,\pi)+\pi(\R)^s\, \widehat{u}(t,\pi)=\widehat{f}(t,\pi)\,,\\
	\widehat{u}(0,\pi)=\widehat{u}_0(\pi)\,,\widehat{u}_t(0,\pi)=\widehat{u}_1(\pi)\,,
	\end{cases}
\end{equation}
where $\widehat{f}(t,\pi)$ denotes the group Fourier transform of the function $f(t,x):=-m(x)u(t,x)$. Taking into account the matrix representation of $\pi(\R)$, we rewrite the matrix equation \eqref{prop1.ft} componentwise as the infinite system of equations of the form 
\begin{equation}
	\label{prop1.ft2}
	\widehat{u}_{tt}(t,\pi)_{k,l}+\pi^{2s}_{k}\cdot \widehat{u}(t,\pi)_{k,l}=\widehat{f}(t,\pi)_{k,l}\,,
\end{equation}
with initial conditions $\widehat{u}(0,\pi)_{k,l}=\widehat{u}_0(\pi)_{k,l}$ and $\widehat{u}_t(0,\pi)_{k,l}=\widehat{u}_1(\pi)_{k,l}$, for all $\pi \in \widehat{\G}$ and for any $k,l \in \mathbb{N}$, where now $\widehat{f}(t,\pi)_{k,l}$ can be regarded as the source term of the second order differential equation as in \eqref{prop1.ft2}.
\par  Now, let us decouple the matrix equation in \eqref{prop1.ft2} by fixing $\pi \in \widehat{\G}$, and treat each of the equations represented in \eqref{prop1.ft2} individually. If we denote by 
\[
v(t):=\widehat{u}(t,\pi)_{k,l}\,,\beta^{2s}:=\pi^{2s}_{k}\,,f(t):=\widehat{f}(t,\pi)_{k,l}\,,
\]  
and 
\[
v_0:=\widehat{u}_0(\pi)_{k,l}\,,v_1:=\widehat{u}_1(\pi)_{k,l}\,,
\]
then \eqref{prop1.ft2} becomes
\begin{equation}\label{prop1.fix.pi}
	\begin{cases}
	v^{''}(t)+\beta^{2s}\cdot v(t)=f(t)\,,\\
	v(0)=v_0\,,v^{'}(0)=v_1\,,
	\end{cases}
\end{equation}
with $\beta>0.$
By solving first the homogeneous version of \eqref{prop1.fix.pi}, and then by applying Duhamel's principle (see e.g. \cite{Eva98}), we get the following representation of the solution of \eqref{prop1.fix.pi}
\begin{eqnarray}
	\label{prop1.solut.fix.pi}
	v(t)&=&\cos(t \beta^s)v_0+\frac{\sin(t \beta^s)}{\beta^s}v_1+\int_{0}^{t}\frac{\sin((t-s)\beta^s)}{\beta^s}f(s)\,ds. \nonumber\\
\end{eqnarray}
Assuming without loss of generality that $T \geq 1$, and using the estimates
\[
|\cos(t\beta^s)|\leq 1\,,\quad \forall t \in [0,T]\,,
\]
and 
\[
|\sin(t\beta^s)|\leq 1\,,
\]
for large values of the quantities $t\beta^s$, while for small values of them, the estimates
\[
|\sin(t\beta^s)|\leq t \beta^s\leq T\beta^s\,,
\] 
inequality \eqref{prop1.solut.fix.pi} yields
\begin{eqnarray*}
|v(t)| &\leq & |v_0|+T|v_1|+\|t-s\|_{L^{2}[0,T]}\|f(t)\|_{L^2[0,T]} \lesssim |v_0|+|v_1| +\|f(t)\|_{L^2[0,T]}\,,
\end{eqnarray*}
where we have applied the Cauchy-Schwarz inequality.
Now the last estimate, if substituting back our initial functions in $t$, gives
\[
|\widehat{u}(t,\pi)_{k,l}|^2\lesssim |\widehat{u}_0(\pi)_{k,l}|^2+|\widehat{u}_1(\pi)_{k,l}|^2+\|\widehat{f}(t,\pi)_{k,l}\|_{L^2[0,T]}^2\,,
\]
where the latter holds uniformly in $\pi \in \widehat{\G}$ and for each $k,l \in \mathbb{N}$. Recall that for any Hilbert-Schmidt operator $A$, one has 
\[
\|A\|^{2}_{\textnormal{HS}}=\sum_{k,l}|\langle A \varphi_{k},\varphi_{l}\rangle|^2\,,
\]
for any orthonormal basis $\{\varphi_1,\varphi_2,\cdots\}$, summing the above over $k,l$ we get 
\[
\|\widehat{u}(t,\pi)_{k,l}\|^{2}_{\textnormal{HS}}\lesssim \|\widehat{u}_0(\pi)_{k,l}\|^{2}_{\textnormal{HS}}+\|\widehat{u}_1(\pi)_{k,l}\|^{2}_{\textnormal{HS}}+\sum_{k,l}\int_{0}^{T}|\widehat{f}(t,\pi)_{k,l}|^2\,dt\,.
\]
Next we integrate the last inequality with respect to the Plancherel measure $\mu$ on $\widehat{\G}$, so that using the Plancherel identity \eqref{planc.id}, we obtain 
\begin{equation}\label{before.Pl}
	\|u(t,\cdot)\|\L \lesssim  \|u_0\|\L+\|u_1\|\L+\int_{\G} \sum_{k,l}\int_{0}^{T}|\widehat{f}(t,\pi)_{k,l}|^2\,dt\,d\mu(\pi)\,,
\end{equation}
and if we use Lebesgue's dominated convergence theorem, Fubini's theorem and the Plancherel formula we have
\begin{equation}\label{F,dc,Pl}
\int_{\G} \sum_{k,l}\int_{0}^{T}|\widehat{f}(t,\pi)_{k,l}|^2\,dt\,d\mu=\int_{0}^{T}\int_{\G}\sum_{k,l}|\widehat{f}(t,\pi)_{k,l}|^2\,d\mu\,dt=\int_{0}^{T}\|f(t,\cdot)\|_{L^2(\G)}^{2}\,dt\,.
\end{equation}
Now, by \eqref{prop.1.est.m}, and the formula of $f$ we have
\begin{eqnarray}\label{est.for.f.mu}
\|f(t,\cdot)\|_{L^2(\G)}^{2} & = & \|m(\cdot)u(t,\cdot)\|\L\nonumber\\
& \leq & \|m\|_{L^{\infty}(\G)}\|\sqrt{m}(\cdot)u(t,\cdot)\|\L\nonumber\\
& \lesssim & (1+\|m\|_{L^\infty(\G)})^2\{\|u_1\|\L+\|u_0\|^{2}_{H^{\frac{s \nu}{2}}(\G)}\}\,.
\end{eqnarray}
Combining the inequalities \eqref{before.Pl}, \eqref{F,dc,Pl} and \eqref{est.for.f.mu} we get 
\begin{eqnarray}\label{prop1.est.B}
	\|u(t,\cdot)\|\L & \lesssim &  (1+\|m\|_{L^\infty(\G)})^2\{\|u_1\|\L+\|u_0\|^{2}_{H^{\frac{s \nu}{2}}(\G)}\}\,,
	\end{eqnarray}
uniformly in $t \in [0,T]$. The claim \eqref{prop1.claim} now follows by \eqref{forRs2} and \eqref{prop1.est.B}. Finally, the uniqueness of $u$ is an immediate consequence of \eqref{prop1.claim}, and the proof is complete. 
\end{proof}
\begin{proposition}\label{prop2}
Assume that $Q >\nu s$, and let $m \in L^{\frac{2Q}{\nu s}}(\G)\cap L^{\frac{Q}{\nu s}}(\G)$, $m \geq 0$. If we suppose that $(u_0,u_1) \in H^{\frac{s \nu}{2}}(\G)\times L^2(\G)$, then there exists a unique solution $u \in C([0,T];H^{\frac{s \nu}{2}}(\G))\cap C^{1}([0,T];L^{2}(\G))$ to the Cauchy problem \eqref{kg.eq} satisfying the estimate 
\begin{eqnarray}
\label{prop2.claim}
\lefteqn{\|u(t,\cdot)\|_{H^{\frac{s \nu}{2}}(\G)}+\|\partial_{t}u(t,\cdot)\|_{L^2(\G)}}\nonumber\\
& \lesssim & \left(1+\|m\|_{L^{\frac{2Q}{\nu s}}(\G)}\right) \left(1+\|m\|_{L^{\frac{Q}{\nu s}}(\G)}\right)^{\frac{1}{2}} \left\{\|u_1\|_{L^2(\G)}+\|u_0\|_{H^{\frac{s \nu}{2}}(\G)}\right\}  \,,
\end{eqnarray}
 uniformly in $t \in [0,T]$.
\end{proposition}
\begin{proof}
Proceeding as in the proof of Proposition \ref{prop1.clas.kg}, we have 
\begin{equation}\label{prop2.et}
E(t)=E(0)\,,\quad \forall t \in [0,T]\,,
\end{equation}
where the energy estimate $E$ is given by 
\[
E(t)=\|u_t(t,\cdot)\|\L+\|\R^{\frac{s}{2}}u(t,\cdot)\|\L+\|\sqrt{m}(\cdot)u(t,\cdot)\|\L\,.
\]
Now, applying H\"older's inequality, we get 
\begin{equation}
\label{prop2.Hold}
 \|\sqrt{m} u_0\|\L \leq \|m\|_{L^{q^{'}}(\G)}\|u_{0}\|^{2}_{L^{2q}(\G)},
\end{equation}
where $1<q,q^{'}<\infty$, and $(q,q^{'})$ conjugate exponents, to be chosen later. Observe that if we apply \eqref{inclusions} for $u_0 \in H^{\frac{s \nu}{2}}(\G)$, $b=\frac{s \nu}{2}$, $a=0$, and $q_0=\frac{2Q}{Q-\nu s}$, then $\tilde{q}_0=2$, and we have 
\begin{equation}
\label{prop2.hold1}
\|u_0\|_{L^{q_{0}}(\G)} \lesssim \|\R^{\frac{s}{2}}u_0\|_{L^2(\G)}<\infty\,.
\end{equation}
Choosing $2q=q_0$ in \eqref{prop2.Hold} so that $q=\frac{Q}{Q-\nu s}$, we get $q^{'}=\frac{Q}{\nu s}$, so that 
\begin{equation}
\label{emb}
\|\sqrt{m}u_0\|\L \lesssim \|m\|_{L^{\frac{Q}{\nu s}}(\G)}\|\R^{\frac{s}{2}}u_0\|_{L^2(\G)}^{2}<\infty\,,
\end{equation}
and by \eqref{prop2.et} we can estimate
\begin{eqnarray}
\label{prop2.est.m}
\|\sqrt{m}(\cdot)u(t,\cdot)\|\L  & \leq & \|u_1\|\L+\|u_0\|_{H^{\frac{s \nu}{2}}(\G)}^{2}+\|\sqrt{m}u_0\|\L\nonumber\\
& \lesssim & \|u_1\|\L+\|u_0\|_{H^{\frac{s \nu}{2}}(\G)}^{2}+\|m\|_{L^{\frac{Q}{\nu s}}(\G)}\|u_0\|^{2}_{H^{\frac{s \nu}{2}}(\G)}\nonumber\\
& \leq & \left(1+\|m\|_{L^{\frac{Q}{\nu s}}(\G)}\right)\left\{\|u_1\|\L+\|u_0\|_{H^{\frac{s \nu}{2}}(\G)}^{2}\right\}\,,
\end{eqnarray}
 uniformly in $t \in [0,T]$.
Additionally, \eqref{prop2.et}, under the estimate \eqref{prop2.est.m}, implies
\begin{equation}
\label{prop2.other.est}
\|u_t(t,\cdot)\|\L\,,\|\R^{\frac{s}{2}}u(t,\cdot)\|\L \lesssim \left(1+\|m\|_{L^{\frac{Q}{\nu s}}(\G)}\right)\left\{\|u_1\|\L+\|u_0\|_{H^{\frac{s \nu}{2}}(\G)}^{2}\right\}.
\end{equation}
To show our claim \eqref{prop2.claim}, it suffices to show the desired estimate for the solution norm $\|u(t,\cdot)\|_{L^2(\G)}$. First we observe that by the Sobolev embeddings \eqref{inclusions} and  H\"older's inequality, using \eqref{emb} with $m$ instead of $\sqrt{m}$, 
and $\|m^2\|_{L^{\frac{Q}{\nu s}}(\G)}=\|m\|^2_{L^{\frac{2Q}{\nu s}}(\G)}$,
one obtains 
\[
\|mu(t,\cdot)\|\L \lesssim \|m\|^{2}_{L^{\frac{2Q}{\nu s}}}\|\R^{\frac{s}{2}}u(t,\cdot)\|\L\,,
\]
where the last combined with \eqref{prop2.other.est} yields
\begin{equation}\label{prop2.est.hold2}
\|mu(t,\cdot)\|\L \lesssim \|m\|^{2}_{\frac{2Q}{\nu s}} \left(1+\|m\|_{L^{\frac{Q}{\nu s}}(\G)}\right)\left\{\|u_1\|\L+\|u_0\|_{H^{\frac{s \nu}{2}}(\G)}^{2}\right\}\,.
\end{equation}
Finally, using arguments similar to those we developed in Proposition \ref{prop1.clas.kg}, together with the estimate \eqref{prop2.est.hold2} we get
\begin{eqnarray*}
\label{prop2.est.f}
\|u(t,\cdot)\|\L & \leq & \|u_0\|\L+\|u_1\|\L+\|m(\cdot)u(t,\cdot)\|\L\nonumber\\
& \lesssim & \left\{\|u_1\|\L+\|u_0\|_{H^{\frac{s \nu}{2}}(\G)}^{2}\right\}  \left\{1+\|m\|^{2}_{\frac{2Q}{\nu s}} \left(1+\|m\|_{L^{\frac{Q}{\nu s}}(\G)}\right)\right\}\,,
\end{eqnarray*}
 uniformly in $t \in [0,T]$. The uniqueness of $u$ is immediate by the estimate \eqref{prop2.claim}, and this finishes the proof of Proposition \ref{prop2}.
\end{proof}

\section{Existence and uniqueness of the very weak solution}
Here, we consider the case where the mass-term in \eqref{kg.eq} satisfies some moderateness properties. The latter is satisfied in cases where, for instance, $m$ has strong singularities, namely when  $m=\delta$ or $\delta^2$. This follows by Proposition \ref{prop.mol} for $\delta$, while we can understand $\delta^2$ as an approximating family or in the Colombeau sense.

\begin{definition}[Moderateness]
	\begin{enumerate}
		\item Let $X$ be a normed space of functions on $\G$.  A net of functions $(f_\epsilon)_\epsilon \in X$ is said to be \textit{$X$-moderate} if there exists $N \in \mathbb{N}$ such that \[
		\|f_\epsilon\|_{X}\lesssim \epsilon^{-N}\,,
		\]
		uniformly in $\epsilon \in (0,1]$.
		\item A net of functions $(u_\epsilon)_\epsilon$ in $C([0,T]; H^{\frac{s \nu}{2}}(\G))\cap C^{1}([0,T];L^2(\G))$ is said to be \textit{$C([0,T]; H^{\frac{s \nu}{2}}(\G))\cap C^{1}([0,T];L^2(\G))$-moderate}, or for brevity, \textit{$C^{1}$-moderate}, if there exists $N \in \mathbb{N}$ such that 
		\[
		\sup_{t \in [0,T]}\{\|u(t,\cdot)\|_{H^{\frac{s \nu}{2}}(\G)}+\|\partial_{t}u(t,\cdot)\|_{L^2(\G)}\}\lesssim \epsilon^{-N}\,,
		\]
		uniformly in $\epsilon \in (0,1]$.
	\end{enumerate}
\end{definition}

\begin{definition}[Negligibility]
\label{defn.negl}
{Let $Y$ be a normed space of functions on $\G$.
Let $(f_\epsilon)_\epsilon $, $(\tilde{f}_\epsilon)_\epsilon$ be two nets. Then, the net $(f_\epsilon-\tilde{f}_\epsilon)_\epsilon$ is called $Y$-negligible, if the following condition is satisfied
    \begin{equation}\label{def.cond.negl}
    \|f_\epsilon-\tilde{f}_\epsilon\|_{Y}\lesssim \epsilon^k\,,
    \end{equation}
    for all $k \in \mathbb{N}$, $\epsilon \in (0,1]$. In the case where $f=f(t,x)$ is a function also depending on $t \in [0,T]$, then the negligibility condition \eqref{def.cond.negl} can be regarded as
    \[
    \|f_\epsilon(t,\cdot)-\tilde{f}_\epsilon(t,\cdot)\|_{Y}\lesssim \epsilon^k\,,\quad \forall k \in \mathbb{N}\,,
    \]
    uniformly in $t \in [0,T]$.} The constant in the inequality \eqref{def.cond.negl} can depend on $k$ but not on $\epsilon$.
\end{definition}

In Definitions \ref{defn.vws.kg} and \ref{defn.uniq.kg}, we introduce the notion of the unique very weak solution to the Cauchy problem \eqref{kg.eq}. Our definitions are similar to the one introduced in \cite{GR15}, but here we measure moderateness and negligibility in terms of $L^p(\G)$ or $C^1$-seminorms rather than in terms of Gevrey-seminorms.
\begin{definition}[Very weak solution]\label{defn.vws.kg}
Let $(u_0,u_1) \in H^{\frac{s \nu}{2}}(\G)\times L^2(\G)$. Then, if there exists a non-negative $L^{\infty}(\G)$-moderate (or a $L^{\frac{2Q}{\nu s}}(\G)\cap L^{\frac{Q}{\nu s}}(\G)$-moderate, if we additionally require $Q> \nu s$) {\em approximating net} $(m_\epsilon)_\epsilon$, $m_\epsilon\geq 0$, to $m$,  so that the family $(u_\epsilon)_\epsilon \in C([0,T]; H^{\frac{s \nu}{2}}(\G))\cap C^{1}([0,T]; L^2(\G))$ which solves the $\epsilon$-parametrised problem 
	\begin{equation}\label{kg.reg}
		\begin{cases}
			\partial_{t}^{2}u_\epsilon(t,x) +\mathcal{R}^{s}u_\epsilon (t,x)+m_\epsilon (x)u_\epsilon(t,x)=0\,,\quad (t,x)\in [0,T]\times \mathbb{G},\\
			u_\epsilon(0,x)=u_{0}(x),\; \partial_{t}u_\epsilon(0,x)=u_{1}(x),\; x \in \G\,,	
		\end{cases}       
	\end{equation}
for all $\epsilon \in (0,1]$, is $C^{1}$-moderate, then net $(u_\epsilon)_\epsilon$ is said to be a \textit{very weak solution to the Cauchy problem \eqref{kg.eq}}.

We will also use a modification of this notion for data in the space $\mathcal{E}^{'}(\G)$ of compactly supported distributions. 
Thus, let $u_0,u_1\in \mathcal{E}^{'}(\G)$ or as in Remark \ref{rem:datas}.
Then, if there exists a non-negative $L^{\infty}(\G)$-moderate (or a $L^{\frac{2Q}{\nu s}}(\G)\cap L^{\frac{Q}{\nu s}}(\G)$-moderate, if we additionally require $Q> \nu s$) {\em approximating net} $(m_\epsilon)_\epsilon$, $m_\epsilon\geq 0$, to $m$, and moderate regularisations  $(u_{0,\epsilon},u_{1,\epsilon}) \in H^{\frac{s \nu}{2}}(\G)\times L^2(\G)$ of $(u_0,u_1)$,  so that the family $(u_\epsilon)_\epsilon \in C([0,T]; H^{\frac{s \nu}{2}}(\G))\cap C^{1}([0,T]; L^2(\G))$ which solves the $\epsilon$-parametrised problem 
	\begin{equation}\label{kg.reg2}
		\begin{cases}
			\partial_{t}^{2}u_\epsilon(t,x) +\mathcal{R}^{s}u_\epsilon (t,x)+m_\epsilon (x)u_\epsilon(t,x)=0\,,\quad (t,x)\in [0,T]\times \mathbb{G},\\
			u_\epsilon(0,x)=u_{0,\epsilon}(x),\; \partial_{t}u_\epsilon(0,x)=u_{1,\epsilon}(x),\; x \in \G\,,	
		\end{cases}       
	\end{equation}
for all $\epsilon \in (0,1]$, is $C^{1}$-moderate, then net $(u_\epsilon)_\epsilon$ is said to be a \textit{very weak solution to the Cauchy problem \eqref{kg.eq}}.

We note that we can also combine both cases into one, since for $(u_0,u_1) \in H^{\frac{s \nu}{2}}(\G)\times L^2(\G)$ we can take the trivial regularisations 
$u_{0,\epsilon}=u_0$ and $u_{1,\epsilon}=u_1$ for all $\epsilon>0.$
Consequently, we can also adapt this definition for the data
$(u_0,u_1)\in H^{\frac{s\nu}{2}}(\G)\times \mathcal{E}^{'}(\G)$ or for
$(u_0,u_1) \in \mathcal{E}^{'}(\G)
\times L^2(\G)$.

\end{definition}

\begin{remark} In Definition \ref{defn.vws.kg} above we ask for $m_\epsilon$ to {\em approximate} $m$, to allow for more flexibility. This should be understood as follows. If $m\in \mathcal{D}'(\G)$ is a distribution, we can understand it as a {\em regularisation}, namely, the assumption in Definition \ref{defn.vws.kg} is that there is a Friedrichs mollifier $\psi\geq 0$ such that $m_\epsilon=m*\psi_\epsilon.$
However, the word approximation allows for more flexibility, for example, we can in principle generate an approximating family with a net $\tilde{m}_\epsilon$ such that the one we will discuss in \eqref{mex}. Moreover, this context allows us to start with $m$ being more singular than a distribution: for example, if $m=\delta^2$ we can think of an approximating family $m_\epsilon=\psi_\epsilon^2$. See also Remark \ref{rem1p} for a continuation of this discussion.
\end{remark}

We now formulate the very weak existence result, corresponding to two possibilities of regularising with families $(m_\epsilon)_\epsilon$ with different properties, corresponding to the existence results in Proposition \ref{prop1.clas.kg} and Proposition \ref{prop2}.
\begin{theorem}
	\label{thm.ex.kg}
	Let $(u_0,u_1) \in H^{\frac{s \nu}{2}}(\G)\times L^2(\G)$. 
Then, the Cauchy problem \eqref{kg.eq} has a very weak solution.
\end{theorem}
\begin{proof}
Let $u_0,u_1$ be as in the hypothesis. If $(m_\epsilon)_\epsilon$ is $L^{\infty}(\G)$-moderate (or $L^{\frac{2Q}{\nu s}}(\G)\cap L^{\frac{Q}{\nu s}}(\G)$-moderate), then, since $m_\epsilon \geq 0$, by using \eqref{prop1.claim} (or \eqref{prop2.claim}, respectively) we get 
	\[
\|u_\epsilon(t,\cdot)\|_{H^{\frac{s \nu}{2}}(\G)}+\|\partial_{t}u_\epsilon(t,\cdot)\|_{L^2(\G)} \lesssim \epsilon^{-N}\,,\quad N \in \mathbb{N}\,,
	\]
	for all $t \in [0,T]$ and for any $\epsilon \in (0,1]$. This means that the family of solutions $(u_\epsilon)_\epsilon$ is $C^1$-moderate, and completes the proof of Theorem \ref{thm.ex.kg}.
\end{proof}

The uniqueness of the very weak solution to the Cauchy problem \eqref{kg.eq} can be understood as if a negligible change of the net $(m_\epsilon)_\epsilon$ does not affect the asymptotic behaviour of the family of the very weak solutions. 
In other words, negligible changes of the approximation $m_\epsilon$ of $m$ lead to negligible changes in the solution family $u_\epsilon$, with an appropriate choices of norms to understand the negligibility.
Formally, we have the following definition.
\begin{definition}\label{defn.uniq.kg}
{Let $X$ and $Y$ be normed spaces of functions on $\G$.
We say that the Cauchy problem \eqref{kg.eq} has an $(X,Y)$-unique very weak solution, if for all $X$-moderate nets
	$m_\epsilon\geq 0, \tilde{m}_\epsilon\geq 0,$ such that $(m_\epsilon-\tilde{m}_\epsilon)_\epsilon$ is $Y$-negligible}, it follows that
	\[
	\|u_\epsilon(t,\cdot)-\tilde{u}_\epsilon(t,\cdot)\|_{L^2(\G)} \leq C_N \epsilon^N\,,\quad \forall N \in \mathbb{N}\,,
	\]
	 uniformly in $t \in [0,T]$, and for all $\epsilon \in (0,1]$, where $(u_\epsilon)_\epsilon$ and $(\tilde{u}_\epsilon)_\epsilon$ are the families of solutions corresponding to the $\epsilon$-parametrised problems 
	 \begin{equation}\label{kg.reg.tild}
		\begin{cases}
			\partial_{t}^{2}u_\epsilon(t,x) +\mathcal{R}^{s}u_\epsilon (t,x)+m_\epsilon(x)u_\epsilon(t,x)=0\,,\quad (t,x)\in [0,T]\times \mathbb{G},\\
			u_\epsilon(0,x)=u_{0,\epsilon}(x),\;\partial_{t}u_\epsilon(0,x)=u_{1,\epsilon}(x),\; x \in \G\,,	
		\end{cases}       
	\end{equation}
	and 
\begin{equation}\label{kg.ret.notil}
		\begin{cases}
			\partial_{t}^{2}\tilde{u}_\epsilon(t,x) +\mathcal{R}^{s}\tilde{u}_\epsilon (t,x)+\tilde{m}_\epsilon(x)\tilde{u}_\epsilon(t,x)=0\,,\quad (t,x)\in [0,T]\times \mathbb{G},\\
			\tilde{u}_\epsilon(0,x)=\tilde{u}_{0,\epsilon}(x),\;\partial_{t}\tilde{u}_\epsilon(0,x)=\tilde{u}_{1,\epsilon}(x),\; x \in \G\,,	
		\end{cases}       
	\end{equation}
	respectively.	
\end{definition}

{
\begin{remark}\label{rem1p}
We note that in Definition 3 in the previous paper \cite{ARST21a}, the word `regularisation' needs to be understood, in general, as an approximation not necessarily depending on the classical convolution and specific mollifiers. In this case, our definition of the uniqueness of the very weak solutions here includes also the version in Definition 3 in \cite{ARST21a}, but Definition \ref{defn.uniq.kg} makes it more rigorous.
To clarify this further, we can take, for example, $m_\epsilon$ to be a regularisation of $m$ by a convolution (if $m$ is a distribution), and take 
\begin{equation}\label{mex}
\tilde{m}_\epsilon=m_\epsilon+e^{-1/\epsilon}.
	\end{equation}
Then the net 
$(m_\epsilon-\tilde{m}_\epsilon)_\epsilon$ is $L^\infty$-negligible, and so it is suitable to be used in Definition \ref{defn.uniq.kg}. If $m=\delta^2$, we can take $m_\epsilon=\psi_\epsilon^2$ for a Friedrichs mollifier $\psi$, and still, for example, $\tilde{m}_\epsilon$ as in
\eqref{mex}. We also note that Definition \ref{defn.uniq.kg} can be also interpreted as stability. In fact, in Definition \ref{defn.uniq.kg} we do not assume $m_\epsilon$ to approximate $m$ since we can prove the required property without this assumption (as in Theorem \ref{thm.uniq.kg.1} and Theorem \ref{thm.uniq.kg.2}). This allows for our results to be applicable to cases like $m=\delta^2$, since with this approach we do not need to explain in which sense $m_\epsilon=\psi_\epsilon^2$ approximates $m=\delta^2$.
\end{remark}
}

We now give some clarification of the moderateness assumption of the regularisations (or appro\-ximations).
{Let us underline that, the global structure of $\mathcal{E}^{'}$-distributions, implies that the assumption on the $L^{p}$-moderateness, for $p \in [1,\infty]$, is natural for nets that arise as regularisations of a compactly supported distribution in $\mathcal{E}^{'}$ via convolutions with a mollifier as in \eqref{mol}.}
 
{ \begin{proposition}\label{prop.mol}
Let $v \in \mathcal{E}^{'}(\G)$, and let 
$v_\epsilon=v*\psi_\epsilon$ be obtained as
the convolution of $v$ with a mollifier $\psi_\epsilon$ as in \eqref{mol}.
Then the regularising net $(v_\epsilon)_\epsilon$ is $L^{p}(\G)$-moderate for any $p \in [1,\infty]$.
\end{proposition}}
{\begin{proof}
Recall, that for $v \in \mathcal{E}^{'}(\G)$ we can find $m \in \mathbb{N}$ and compactly supported continuous functions $f_\beta \in C(\G)$ such that 
\[
v=\sum_{|\beta|\leq m} \partial^{\beta}f_\beta\,,
\]
where $|\beta|$ denoted the length of the multi-index $\beta$. Considering the convolution of $v$ with a mollifier $\psi_\epsilon$ as in \eqref{mol} yields
		\[
		v_\epsilon=	v \ast \psi_{\epsilon}=\left(\sum_{|\beta|\leq m} \partial^{\beta}f_\beta \right) \ast \psi_\epsilon=\sum_{|\beta|\leq m} \left(\partial^{\beta}f_\beta \ast \psi_\epsilon \right)\,,
		\] 
		where each term in the above sum can be rewritten as 
		\begin{eqnarray*}
		\partial^{\beta}f_\beta \ast \psi_\epsilon & = & \langle \partial^{\beta}f_\beta (yx^{-1}),\psi_\epsilon(x)\rangle=(-1)^{|\beta|}\langle f_\beta(yx^{-1}), \partial^{\beta}\psi_\epsilon(x)\rangle\\
		& =& (-1)^{|\beta|}\epsilon^{-Q}\langle f_\beta,\partial^{\beta}\psi(D_{\epsilon^{-1}}(x)\rangle \\
	& =& 	 (-1)^{|\beta|}\epsilon^{-Q-[\beta]}\langle f_\beta,(\partial^{\beta}\psi)(D_{\epsilon^{-1}}(x)\rangle\,,
\end{eqnarray*}
where $[\beta]$ stands for the homogeneous length of $\beta$.

Finally, since $f_\beta, \psi$ are compactly supported, we get $ f_\beta,(\partial^{\beta}\psi)(D_{\epsilon^{-1}}\cdot)\in L^p(\G)$, for all $p$, and this finishes the proof of Proposition \ref{prop.mol}.
\end{proof}
} 
\begin{remark}\label{rem:datas}
It is straightforward to extend Proposition \ref{prop.mol} to its Sobolev spaces version.
Consequently, in view of this, the assumptions of Theorem \ref{thm.ex.kg} can be partially relaxed to $u_0,u_1\in \mathcal{E}^{'}(\G)$, or to 
$(u_0,u_1)\in H^{\frac{s\nu}{2}}(\G)\times \mathcal{E}^{'}(\G)$ or to
$(u_0,u_1) \in \mathcal{E}^{'}(\G)
\times L^2(\G)$, so that Theorem 
\ref{thm.ex.kg} also holds true for
$(u_0,u_1) \in \{H^{\frac{s\nu}{2}}(\G) \cup \mathcal{E}^{'}(\G)\} \times \{L^2(\G) \cup \mathcal{E}^{'}(\G)\}$.
Similar assumptions can also be made in Theorems \ref{thm.uniq.kg.1} and \ref{thm.uniq.kg.2} below, that show the uniqueness of the very weak solution to the Cauchy problem \eqref{kg.eq}.
\end{remark}

\begin{theorem}\label{thm.uniq.kg.1}
	 Suppose that $(u_0,u_1) \in \{H^{\frac{s\nu}{2}}(\G) \cup \mathcal{E}^{'}(\G)\} \times \{L^2(\G) \cup \mathcal{E}^{'}(\G)\}$.
	 {Then, the very weak solution to the Cauchy problem \eqref{kg.eq} is
$(L^\infty(\G), L^\infty(\G))$-unique.}
\end{theorem}
\begin{proof}
Let  $(u_\epsilon)_\epsilon$ and $(\tilde{u}_\epsilon)_\epsilon$ be the families of solutions corresponding to the Cauchy problems \eqref{kg.reg.tild} and \eqref{kg.ret.notil}, respectively. If we denote by $U_\epsilon(t,\cdot):=u_\epsilon(t,\cdot)-\tilde{u}_\epsilon(t,\cdot)$, then $U_\epsilon$ satisfies
	\begin{equation}\label{kg.uniq}
		\begin{cases}
			\partial_{t}^{2}U_\epsilon(t,x) +\mathcal{R}^{s}U_\epsilon (t,x)+m_\epsilon(x)U_\epsilon(t,x)=f_\epsilon(t,x)\,,\quad (t,x)\in [0,T]\times \mathbb{G},\\
			U_\epsilon(0,x)=0\,,\partial_{t}U_\epsilon(0,x)=0\,, x \in \G\,,	
		\end{cases}       
	\end{equation}
where $f_\epsilon(t,x):=(\tilde{m}_\epsilon(x)-m_\epsilon(x))\tilde{u}_\epsilon(t,x)$. \\

The solution of the Cauchy problem \eqref{kg.uniq} can be expressed in terms of the solution to the corresponding homogeneous Cauchy problem using Duhamel's principle. Indeed, if for a fixed $\sigma$, $V_\epsilon(t,x;\sigma)$ is the solution of the homogeneous problem 
\begin{equation}\label{hom.Ve}
\begin{cases*}
	\partial_{t}^{2}V_\epsilon(t, x;\sigma)+\R^s V_\epsilon(t, x;\sigma)+m_\epsilon V_\epsilon(t, x;\sigma)=0\,,\quad\text{in}\, (\sigma,T] \times \G,\\
	V_\epsilon(t, x;\sigma)=0\,,\quad \partial_{t}V_\epsilon(t, x;\sigma)=f_\epsilon(\sigma,x), \,\quad \text{on}\,\{t=\sigma\}\times \G\,,
\end{cases*}
\end{equation}
then $U_\epsilon$ is given by $U_\epsilon(t,x)=\int_{0}^{t}V_\epsilon(t,x;\sigma)\,d\sigma$. 

Since by Minkowski's integral inequality we know
\[
\| \int_{0}^{t}V_{\epsilon}(t, \cdot;\sigma)\,d\sigma \|_{L^2(\G)}\leq \int_{0}^{t}\|V_\epsilon(t, \cdot;\sigma)\|_{L^2(\G)}\,d\sigma\,,
\]
using the energy estimate \eqref{prop1.claim} to control $L^2(\G)$-norm of the solution  $V_\epsilon$ to the homogeneous problem \eqref{hom.Ve}, and subsequently of $U_\epsilon$, we get
 \begin{eqnarray*}
  \|U_\epsilon(t,\cdot)\|_{L^2(\G)} & \leq & \int_{0}^{T}\|V_\epsilon(t, \cdot;\sigma)\|_{L^2(\G)}\,d\sigma\\
  & \lesssim & (1+\|m_\epsilon\|_{L^{\infty}(\G)})\int_{0}^{T}\|f_\epsilon(\sigma,\cdot)\|_{L^2(\G)}\,d\sigma\\
		& \lesssim & (1+\|m_\epsilon\|_{L^{\infty}(\G)}) \|\tilde{m}_\epsilon-m_\epsilon\|_{L^{\infty}(\G)}\int_{0}^{T}\|\tilde{u}_\epsilon(\sigma,\cdot)\|_{L^2(\G)}\,d\sigma\,,
 \end{eqnarray*} 
where we use the estimate 
\[
\|f_\epsilon(\sigma,\cdot)\|_{L^2(\G)}=\|(\tilde{m}_\epsilon-m_\epsilon)(\cdot)\tilde{u}_\epsilon(\sigma,\cdot)\|_{L^2(\G)}\leq  \|\tilde{m}_\epsilon-m_\epsilon\|_{L^{\infty}(\G)}\|\tilde{u}_\epsilon(\sigma,\cdot)\|_{L^2(\G)}\,.
\]
Now, using the fact that $(m_\epsilon)_\epsilon$ is $L^{\infty}(\G)$-moderate, while also that the net $(\tilde{u}_\epsilon)_\epsilon$, as being a very weak solution to the Cauchy problem \eqref{kg.reg.tild}, is $C^1$-moderate and that $(m_\epsilon-\tilde{m}_\epsilon)_\epsilon$ is $L^{\infty}$-negligible, we get that 
\[
\|U_\epsilon(t,\cdot)\|_{L^2(\G)} \lesssim \epsilon^{-N_1+N}  \int_{0}^{T}\epsilon^{-N_2}\,d\sigma = T \, \epsilon^{-N_1-N_2+N} \,,
\]
for some $N_1,N_2 \in \mathbb{N}$, and for all $N \in \mathbb{N}$, $\epsilon \in (0,1]$. That is, we have
\[
\|U_\epsilon(t,\cdot)\|_{L^2(\G)} \lesssim \epsilon^{k}\,,
\]
for all $k \in \mathbb{N}$, and the last shows that the net $(u_\epsilon)_\epsilon$ is the unique very weak solution to the Cauchy problem \eqref{kg.eq}.
\end{proof}
Alternative to Theorem \ref{thm.uniq.kg.1} conditions on the nets $(m_\epsilon)_\epsilon, (\tilde{m}_\epsilon)_\epsilon$ that guarantee the very weakly well-posedness of \eqref{kg.eq} are given in the following theorem.

\begin{theorem}\label{thm.uniq.kg.2}
Let $Q>\nu s$, and suppose that $(u_0,u_1) \in \{H^{\frac{s\nu}{2}}(\G) \cup \mathcal{E}^{'}(\G)\} \times \{L^2(\G) \cup \mathcal{E}^{'}(\G)\}$. Then, the very weak solution to the Cauchy problem \eqref{kg.eq} is $(L^\infty(\G),L^{\frac{2Q}{\nu s}}(\G))$-unique. Moreover, 
the very weak solution to the Cauchy problem \eqref{kg.eq} is also 
$(L^{\frac{2Q}{\nu s}}(\G)\cap L^{\frac{Q}{\nu s}}(\G),L^{\frac{2Q}{\nu s}}(\G))$-unique
and
$(L^{\frac{2Q}{\nu s}}(\G)\cap L^{\frac{Q}{\nu s}}(\G),L^\infty(\G))$-unique.
\end{theorem}
\begin{proof}
We will only prove the $(L^\infty(\G),L^{\frac{2Q}{\nu s}}(\G))$-uniqueness as the other two uniqueness statements are similar. Proceeding as we did in the proof of Theorem \ref{thm.uniq.kg.1}, we arrive at \begin{multline*}
    \|U_\epsilon(t,\cdot)\|_{L^2(\G)} \lesssim (1+\|m_\epsilon\|_{L^{\infty}(\G)})\int_{0}^{T}\|f_\epsilon(\sigma,\cdot)\|_{L^2(\G)} d \sigma
    \\ =(1+\|m_\epsilon\|_{L^{\infty}(\G)}) \int_{0}^{T}\|(\tilde{m}_\epsilon-m_\epsilon)(\cdot)\tilde{u}_\epsilon(\sigma,\cdot)\|_{L^2(\G)}d \sigma.
 \end{multline*}
for all $t\in[0, T]$.
Additionally, by applying H\"older's inequality, together with the Sobolev embeddings \eqref{inclusions}, we have 
    \[
    \|(\tilde{m}_\epsilon-m_\epsilon)(\cdot)\tilde{u}_\epsilon(t,\cdot)\|_{L^2(\G)} \leq \|\tilde{m}_\epsilon-m_\epsilon\|_{L^{\frac{2Q}{\nu s}}(\G)}\|\R^{\frac{s}{2}}\tilde{u}_\epsilon(t,\cdot)\|_{L^2(\G)}\,,
    \]
    where since $(\tilde{u}_\epsilon)$, as being the very weak solution corresponding to the $L^{\infty}(\G)$-moderate net $(\tilde{m}_\epsilon)_\epsilon$, is $C^1$-moderate, we have 
    \[
    \|\R^{\frac{s}{2}}\tilde{u}_\epsilon(t,\cdot)\|_{L^2(\G)} \lesssim \epsilon^{-N_1}\,,\quad \text{for some}\,N_1 \in \mathbb{N}\,.
    \]
    Summarising the above, and since
    \[
    \|\tilde{m}_\epsilon-m_\epsilon\|_{L^{\frac{2Q}{\nu s}}(\G)}\lesssim \epsilon^{N}\,,\quad \forall N \in \mathbb{N}\,,
    \]we obtain 
    \[
    \|U_\epsilon(t,\cdot)\|_{L^2(\G)}\lesssim \epsilon^k\,,\quad \forall k \in \mathbb{N}\,,
    \]
uniformly in $t$, and this finishes the proof of Theorem \ref{thm.uniq.kg.2}.
\end{proof}

\section{Consistency of the very weak solution with the classical one}

The next theorems stress the conditions, on the coefficient $m$ and on the initial data $u_0,u_1$, under which, the classical solution to the Cauchy problem \eqref{kg.eq} can be recaptured by its very weak solution. In the statements below, we understand the classical solutions as those given by Proposition \ref{prop1.clas.kg} or Proposition \ref{prop2}, depending on the assumptions. By the `regularisations' $m_\epsilon=m*\psi_\epsilon$ below we understand the convolution with non-negative Friedrichs mollifiers $\psi\geq 0$.

\begin{theorem}\label{thm.consis.1}
 Let $Q > \nu s$. Consider the Cauchy problem \eqref{kg.eq}, and let $(u_0,u_1) \in H^{\frac{s \nu}{2}}(\G)\times L^2(\G)$. Assume also that $m \in L^{\frac{2Q}{\nu s}}(\G)\cap L^{\frac{Q}{\nu s}}(\G) $, $m\geq 0$, and that $(m_\epsilon)_\epsilon$, is a regularisation of the coefficient $m$. Then the regularised net $(u_\epsilon)_\epsilon$ converges, as $\epsilon \rightarrow 0$, in  $L^2(\G)$ to the classical solution $u$ given by Proposition \ref{prop2}.
\end{theorem}
\begin{proof}
	Let $u$ be the classical solution of \eqref{kg.eq} given by Proposition \ref{prop2}, and let $(u_\epsilon)$ be the very weak solution of the regularised analogue of it as in \eqref{kg.reg}. Then, we get 
	\[
	\begin{cases*}
		\partial_{t}^{2}(u-u_\epsilon)(t,x)+\R^s(u-u_\epsilon)(t,x)+m_\epsilon(x)(u-u_\epsilon)(t,x)=\eta_\epsilon(t,x), \\
		(u-u_\epsilon)(0,x)=0\,,\partial_{t}(u-u_\epsilon)(0,x)=0\,,
	\end{cases*}
	\]
	where $(t,x)\in [0,T] \times \G$, and $\eta_\epsilon(t,x):=(m_\epsilon(x)-m(x))u(t,x)$. If we denote by $U_\epsilon$ the difference $U_\epsilon(t,x):=(u-u_\epsilon)(t,x)$, the above can be rewritten equivalently as
	\begin{equation}\label{eq.m=0}
		\begin{cases*}
			\partial_{t}^{2}U_\epsilon(t,x)+\R^s U_\epsilon(t,x)+m_\epsilon(x)U_\epsilon(t,x)=\eta_\epsilon(t,x), \\
			U_\epsilon(0,x)=0\,,\partial_{t}U_\epsilon(0,x)=0\,.
		\end{cases*}
	\end{equation}
	
Therefore, if we denote by $W_\epsilon(t, x; \sigma)$ the solution to the corresponding homogeneous problem with the initial data at $\{t=\sigma\}\times \G$
\[
W_\epsilon(t,x;\sigma)=0\,,\quad \text{and}\quad \partial_{t} W_\epsilon(t,x; \sigma)=\eta_\epsilon(\sigma,x)\,,
\] then by Proposition \ref{prop2} we get
\begin{multline*}\label{duh,in.d}
	 \|W_\epsilon(t,\cdot;\sigma)\|_{L^2(\G)}
	  \lesssim  (1+\|m_\epsilon\|_{L^{\frac{2Q}{\nu s}}(\G)}) \left(1+\|m_\epsilon\|_{L^{\frac{Q}{\nu s}}(\G)} \right)^{\frac{1}{2}} \|\eta_\epsilon(\sigma,\cdot)\|_{L^2(\G)}\nonumber\\
	  \leq  (1+\|m_\epsilon\|_{L^{\frac{2Q}{\nu s}}(\G)}) \left(1+\|m_\epsilon\|_{L^{\frac{Q}{\nu s}}(\G)} \right)^{\frac{1}{2}} \|m_\epsilon-m\|_{L^{\frac{2Q}{\nu s}}(\G)}\|\R^{\frac{s}{2}}u(\sigma,\cdot)\|_{L^2(\G)}\,,
	 \end{multline*}
 uniformly in $t \in [\sigma,T]$ and $\sigma \in [0,T]$, where we apply H\"older's inequality and the Sobolev embeddings \eqref{inclusions}. Since $m \in L^{\frac{2Q}{\nu s}}(\G)$, we have $\|m_\epsilon-m\|_{L^{\frac{2Q}{\nu s}}(\G)}\rightarrow 0$, so that  taking the limit of the above as $\epsilon \rightarrow 0$, we get 
 \begin{equation}
     \label{duh.e.0}
      \|W_\epsilon(t,\cdot;\sigma)\|_{L^2(\G)} \rightarrow 0\,,
 \end{equation}
 uniformly in $t \in [\sigma,T]$ and $\sigma \in [0,T]$. Now, Duhamel's principle allows us to express the solution to the inhomogeneous problem with respect to the homogeneous one as  
	 \begin{equation}\label{duh}
	 	U_\epsilon(t,x)=\int_{0}^{t}W_\epsilon(t,x;\sigma)\,d\sigma\,,\end{equation} 
 so that, by \eqref{duh.e.0}, \eqref{duh}, and Minkowski's integral inequality
 \[
 \| \int_{0}^{t}W_{\epsilon}(t,\cdot;\sigma)\,d\sigma \|_{L^2(\G)}\leq \int_{0}^{t}\|W_\epsilon(t,\cdot;\sigma)\|_{L^2(\G)}\,d\sigma\,,
 \]
 we obtain
 \begin{eqnarray*}\label{thm.final.ineq}
 	\|U_\epsilon(t,\cdot)\|_{L^2(\G)}  \leq T \sup_{\sigma \in [0,T]} \|W_\epsilon(t,\cdot; \sigma)\|_{L^2(\G)}\rightarrow 0\,,\quad \text{as}\quad \epsilon \rightarrow 0.
 \end{eqnarray*}
This means that $u_\epsilon \rightarrow u$ with respect to $L^2(\G)$-norm, and this finishes the proof of Theorem \ref{thm.consis.1}.
\end{proof}

In the following theorem we denote by $C_{0}(\G)$ the space of continuous functions on $\G$ vanishing at infinity, that is, such that for every $\epsilon>0$ there exists a compact set $K$ outside of which we have $|f|<\delta$. 

\begin{theorem}\label{thm.consis.2}
Consider the Cauchy problem \eqref{kg.eq}, and let $(u_0,u_1) \in H^{\frac{s \nu}{2}}(\G)\times L^2(\G)$. Assume also that $m \in C_0(\G)$, $m\geq 0$, and that $(m_\epsilon)_\epsilon$, $m_\epsilon\geq 0$, is a regularisation of the coefficient $m$. Then the regularised net $(u_\epsilon)_\epsilon$ converges, as $\epsilon \rightarrow 0$, in $L^2(\G)$ to the classical solution $u$ given by Proposition \ref{prop1.clas.kg}. 
\end{theorem}

Before giving the proof of Theorem \ref{thm.consis.2}, let us make the following observation: If $m \in C_0(\G)$, then $\|m_\epsilon\|_{L^\infty(\G)}\leq C<\infty$, uniformly in $\epsilon \in (0,1]$. 

\begin{proof}[Proof of Theorem \ref{thm.consis.2}]

First observe that for $m$, $(m_\epsilon)_\epsilon$ as in the hypothesis, we have $m_\epsilon \in L^{\infty}(\G)$ for each $\epsilon \in (0,1]$. Now, as in \eqref{eq.m=0}, if we denote by $W_\epsilon$ the solution to the problem 
\begin{equation*}
\begin{cases}
\partial_{t}^{2}W_\epsilon(t, x; \sigma)+\R^s W_\epsilon(t, x; \sigma)+m_\epsilon(x)W_\epsilon(t, x; \sigma)=0\,,\\
W_\epsilon(t, x; \sigma)=0\,,\partial_t W_\epsilon(t, x; \sigma)=\eta_\epsilon(\sigma, x)\quad \text{on}\quad \{t=\sigma\}\times \G\,,
\end{cases}
\end{equation*}
where $\eta_\epsilon(t,x):=(m_\epsilon(x)-m(x))u(t,x)$, then by Proposition \ref{prop1.clas.kg} we obtain
\begin{eqnarray*}
\|W_\epsilon(t,\cdot; \sigma)\|_{L^2(\G)}& \lesssim & (1+\|m_\epsilon\|_{L^{\infty}(\G)})\|\eta_\epsilon(\sigma,\cdot)\|_{L^2(\G)}\nonumber\\
        & \leq & (1+\|m_\epsilon\|_{L^{\infty}(\G)}) \|m_\epsilon-m\|_{L^{\infty}(\G)}\|u(\sigma,\cdot)\|_{L^2(\G)}\,,
\end{eqnarray*}
uniformly in $t \in [\sigma,T]$ and $\sigma \in [0,T]$. Now, by Lemma 3.1.58 in \cite{FR16} we have 
    \[
    \|m_\epsilon-m\|_{L^{\infty}(\G)} \rightarrow 0\,,\quad \text{as}\quad \epsilon \rightarrow 0\,,
    \]
    so that by the above we get 
    \begin{equation}
        \label{thm.cons2.infty}
         \|W_\epsilon(t,\cdot; \sigma)\|_{L^2(\G)}\rightarrow 0\,,\quad \text{as}\quad \epsilon \rightarrow 0\,,
\end{equation}
uniformly in $t \in [\sigma,T]$ and $\sigma \in [0,T]$. Finally, by Duhamel's principle if $U_\epsilon$ is the solution to the non-homogeneous problem \eqref{eq.m=0}, then by \eqref{thm.cons2.infty} we get 
    \[
    \|U_\epsilon(t,\cdot)\|_{L^2(\G)}\rightarrow 0\,,
    \]
    and this completes the proof  of Theorem \ref{thm.consis.2}.
\end{proof}

\begin{remark}\label{finrem}
We note that in Theorem 4 in the paper \cite{ARST21a}, one wrote the assumption that $m \in L^{\infty}(\mathbb R^d)$ in the consistency result. This may be not sufficient in general. Indeed, to be more accurate, it is better to ask $m$ to be in the 
subspace $C_0(\mathbb R^d)$ of $L^{\infty}(\mathbb R^d)$.
In this way we obtain a correction to the statement of Theorem 4 in  \cite{ARST21a} as a special case of Theorem \ref{thm.consis.2} with 
$\G=\mathbb R^d$ and $\mathcal{R}$ being the positive Laplacian $-\Delta$.
\end{remark}

\end{document}